\documentclass[reqno,10pt]{amsart}
\usepackage{graphicx}
\usepackage{amsmath,amsthm,amsfonts,amssymb,latexsym,epsfig}
\newtheorem{theorem}{Theorem}[section]
\newtheorem{lemma}[theorem]{Lemma}

\theoremstyle{definition}

\usepackage{hyperref}
\theoremstyle{remark}

\usepackage{orcidlink}
\usepackage{pdflscape}
\numberwithin{equation}{section}
\usepackage{diagbox}
\usepackage{booktabs} 
\usepackage{array} 
\begin{document}

\title[short text for running head]{``On $7^k$-regular partitions modulo powers of $7$''}
\title{On $7^k$-regular partitions modulo powers of $7$}

\author[D. S. Gireesh and HemanthKumar B.]
{D. S. Gireesh $^{1}$\orcidlink{0000-0002-2804-6479}}
\address{$^{1}$Department of Mathematics, BMS College of Engineering, P.O. Box No.: 1908, Bull Temple Road,
Bengaluru-560 019, Karnataka, India.}
\email{gireeshdap@gmail.com}
\thanks{}

\author[]
{HemanthKumar B.$^{2}$\orcidlink{0000-0001-7904-293X}}

\address{$^2$Department of Mathematics, RV College of Engineering, RV Vidyanikethan Post, Mysore Road, Bengaluru-560 059, Karnataka, India.}
\email{hemanthkumarb.30@gmail.com}
\thanks{}

\date{}
          
\begin{abstract}
In this study, we explore the arithmetic properties of $b_{7^k}(n)$ for any $k\geq1$, which enumerates the partitions of $n$ where no part is divisible by $7^k$. By constructing generating functions for $b_{7^k}(n)$ over specific arithmetic progressions, we establish a collection of Ramanujan-type congruences.
\end{abstract}

\medskip

\maketitle

\section{Introduction}
\label{intro}
A partition of a positive integer $n$ is defined as a non-increasing sequence of positive integers whose sum is $n$. Let $p(n)$ denote the number of partitions of $n$, with its generating function given by
\[\sum\limits_{n\geq0}p(n)q^n=\frac1{f_1}.\]
Throughout this paper, we define
\[f_r:=(q^r;q^r)_\infty=\prod\limits_{m=1}^{\infty}(1-q^{rm}).\]

Ramanujan \cite{Ram2} conjectured that if $\ell\in\{5,7,11\}$ and $0<\delta_{\ell,\beta}<\ell^\beta$ such that $24\delta_{\ell,\beta}\equiv 1\pmod{\ell^\beta}$, then for $n\geq0$,
\begin{equation}\label{RC}
    p\left(\ell^\beta n+\delta_{\ell,\beta}\right)\equiv 0\pmod{\ell^\beta}.
\end{equation}
The above conjecture was proved by Watson \cite{GNW} for the case $\ell=5$. Since $p(243)$ is not divisible by $7^3$, this contradicts the above conjecture for $\ell=7$ and $\beta=3$. This inconsistency was noticed by Chowla\cite{CS}. The appropriate modification of the conjecture was later proved by Watson \cite{GNW}, namely that if $k\geq1$, then
\begin{equation}\label{pc11}
    p(7^{2k-1}n+\lambda_{2k-1})\equiv 0\pmod{7^k},
\end{equation}
and
\begin{equation}\label{pc12}
    p(7^{2k}n+\lambda_{2k})\equiv 0\pmod{7^{k+1}},
\end{equation}
where $\lambda_k$ is the reciprocal modulo $7^k$ of $24$.

Using the classical identities of Euler and Jacobi, Hirschhorn and Hunt \cite{HH} provided a simple proof of \eqref{RC} for $\ell=5$ and a similar proof for the case $\ell=7$ was later given by Garvan\cite{FG} by providing generating functions
\begin{equation}\label{H1}
\sum_{n\geq 0}p\left(7^{2k-1}n+\lambda_{2k-1}\right)q^n=\sum_{j\geq 1}x_{2k-1,j} \, q^{j-1}\frac{f_7^{4j-1}}{f_1^{4j}}
\end{equation}
and
\begin{equation}\label{H2}
\sum_{n\geq 0}p\left(7^{2k}n+\lambda_{2k}\right)q^n=\sum_{j\geq 1}x_{2k,j} \, q^{j-1}\frac{f_7^{4j}}{f_1^{4j+1}},
\end{equation}
where the coefficient vectors $\mathbf{x}_k= (x_{k,1}, x_{k,2},\dots)$ are given by
$$\mathbf{x}_1=\left(x_{1,1},x_{1,2},\dots\right)=\left(7,7^2,0,\dots\right),$$ and for $k\geq1$,
\[
   x_{k+1, i}=\begin{cases}
  \displaystyle \sum_{j\geq1} x_{k, j} \,m_{4j,j+i}, \,\,\text{if}\,\,k\,\,\text{is odd},\\
   \displaystyle \sum_{j\geq1}x_{k, j} \,m_{4j+1,j+i}, \,\,\text{if}\, \,k \,\,\text{is even},
    \end{cases}
\]
where the first seven rows of $M=\left(m_{i,j}\right)_{i,j\geq 1}$ are
\begin{table}[h!]
\centering
\caption{Values of $m_{i,j}$, where $1 \leq i \leq 7$ and $1 \leq j \leq 7$.}
\renewcommand{\arraystretch}{1.5} 
\setlength{\tabcolsep}{4pt} 
\begin{tabular}{|c|ccccccc|}
\hline
\diagbox[width=2em]{$i$}{$j$} & $1$ & $2$ & $3$ & $4$ & $5$ & $6$ & $7$ \\ \hline
1 & $7$ & $7^2$ & $0$ & $0$ & $0$ & $0$ & $0$ \\
2 & $10$ & $9 \times 7^2$ & $2 \times 7^4$ & $7^5$ & $0$ & $0$ & $0$ \\
3 & $3$ & $114 \times 7$ & $85 \times 7^3$ & $24 \times 7^5$ & $3 \times 7^7$ & $7^8$ & $0$ \\
4 & $0$ & $82 \times 7$ & $176 \times 7^3$ & $845 \times 7^4$ & $272 \times 7^6$ & $46 \times 7^8$ & $4 \times 7^{10}$ \\
5 & $0$ & $190$ & $1265 \times 7^2$ & $1895 \times 7^4$ & $1233 \times 7^6$ & $3025 \times 7^7$ & $620 \times 7^9$ \\
6 & $0$ & $27$ & $736 \times 7^2$ & $16782 \times 7^3$ & $20424 \times 7^5$ & $12825 \times 7^7$ & $4770 \times 7^9$ \\
7 & $0$ & $1$ & $253 \times 7^2$ & $1902 \times 7^4$ & $4246 \times 7^6$ & $31540 \times 7^7$ & $19302 \times 7^9$ \\ \hline
\end{tabular}
\end{table}

\begin{table}[h!]
\centering
\caption{Values of $m_{i,j}$, where $1 \leq i \leq 7$ and $8 \leq j \leq 14$.}
\renewcommand{\arraystretch}{1.5} 
\setlength{\tabcolsep}{4pt} 
\begin{tabular}{|c|ccccccc|}
\hline
\diagbox[width=2em]{$i$}{$j$} & $8$ & $9$ & $10$ & $11$ & $12$ & $13$ & $14$ \\ \hline
1 & $0$ & $0$ & $0$ & $0$ & $0$ & $0$ & $0$ \\
2 & $0$ & $0$ & $0$ & $0$ & $0$ & $0$ & $0$ \\
3 & $0$ & $0$ & $0$ & $0$ & $0$ & $0$ & $0$ \\
4 & $7^{11}$ & $0$ & $0$ & $0$ & $0$ & $0$ & $0$ \\
5 & $75 \times 7^{11}$ & $5 \times 7^{13}$ & $7^{14}$ & $0$ & $0$ & $0$ & $0$ \\
6 & $7830 \times 7^{10}$ & $1178 \times 7^{12}$ & $111 \times 7^{14}$ & $6 \times 7^{16}$ & $7^{17}$ & $0$ & $0$ \\
7 & $7501 \times 7^{11}$ & $1944 \times 7^{13}$ & $2397 \times 7^{14}$ & $285 \times 7^{16}$ & $22 \times 7^{18}$ & $7^{20}$ & $7^{20}$ \\ \hline
\end{tabular}
\end{table}
and for $i\geq4$, $m_{i,1}=0$, and for $i\geq 8$, $m_{i,2}=0$ and for $i\geq8$, $j\geq3$,
\begin{align}
\nonumber m_{i,j}&=7m_{i-3,j-1}+35m_{i-2,j-1}+49m_{i-1,j-1}+m_{i-7,j-2}+7m_{i-6,j-2}\\&+21m_{i-5,j-2}+49m_{i-4,j-2}+147m_{i-3,j-2}+343m_{i-2,j-2}+343m_{i-1,j-2}.\label{mij}
\end{align}

For any positive integer $\ell$, a partition is said to be $\ell$-regular if none of its parts are divisible by $\ell$. Let $b_\ell(n)$ denote the number of $\ell$-regular partitions of $n$. The generating function $b_\ell(n)$ is given by
\begin{equation}\label{bl}
    \sum\limits_{n\geq0}b_\ell(n)q^n=\frac{f_\ell}{f_1}.
\end{equation}

Wang \cite{LW} and Adiga et al. \cite{AR} derived congruences for $7$-regular partitions analogous to \eqref{pc11} and \eqref{pc12}. Specifically, for $n,\beta\geq 0$, they proved that
\begin{equation}\label{W1}
b_7\left(7^{2\beta+1}n+\frac{3\cdot7^{2\beta+1}-1}{4}\right)\equiv0\pmod{7^{\beta+1}}.
\end{equation}
Adiga and Ranganath\cite{AR} also proved that
\begin{equation}\label{AR2}
    b_{49}\left(7^{\beta+1}(n+1)-2\right)\equiv0\pmod{7^{\beta+1}},
\end{equation}
for all $n,\beta\geq 0$. For further exploration of $\ell$-regular partitions, see references \cite{AB}--\cite{FP}, \cite{HSZ}, \cite{P}, and \cite{w}.

We establish congruences for $7^k$--regular partitions for all positive integers $k$. The congruences previously derived by Wang and Adiga et al. can be seen as specific cases of the more general results presented in this work.

The main results are as follow:
\begin{theorem}\label{th1}
For each $n,\beta \geq0$, and $k\geq1$, we have
\begin{equation}\label{c1}
b_{7^{2k-1}}\left(7^{2k+2\beta-1}n+\frac{18\cdot7^{2k+2\beta-1}-7^{2k-1}+1}{24}\right)\equiv 0 \pmod{7^{k+\beta}},
\end{equation}
and
\begin{equation}\label{c3}
b_{7^{2k}}\left(7^{2k+\beta-1}\left(n+1\right)-\frac{7^{2k}-1}{24}\right)\equiv 0\pmod{7^{k+\beta}}.
\end{equation}
\end{theorem}

If we put $k=1$ in \eqref{c1} and \eqref{c3}, we obtain congruences \eqref{W1} and \eqref{AR2}, respectively.

\section{Preliminaries}
In this section, we present several lemmas that serve as foundational tools for establishing our main results.
Let $H$ be the “huffing” operator modulo $7$, that is,
\[H\left(\sum{a(n)q^n}\right)=\sum{a(7n)q^{7n}}.\]
\begin{lemma}[\cite{FG}, Lemma 3.1 and 3.4]
 If $\xi= \frac{f_1}{q^2 f_{49}}$ and $T = \frac{f_7^4}{q^7 f_{49}^4}$, then
\begin{equation}\label{Gu1}
  H \left( \xi^{-i} \right)  = \sum_{j\geq1} m_{i,j} T^{-j}
\end{equation}
and
\begin{equation}\label{Gu2}
  H \left( \xi^{-4i} \right)  = \sum_{j\geq1} m_{4i,i+j} T^{-i-j}. 
\end{equation}
\end{lemma}
\begin{lemma}
For all $i\geq1$, we have
\begin{align}
H \left( q^{i-2} \frac{f_{7}^{4i-1}}{f_1^{4i-1}} \right) &= \sum_{j\geq1} m_{4i-1,i+j-1} q^{7j-7} \frac{f_{49}^{4j-3}}{f_7^{4j-3}},\label{H4i-1} \\
H \left( q^{i+1} \frac{f_{7}^{4i-3}}{f_1^{4i-3}} \right) &= \sum_{j\geq1} m_{4i-3,i+j-1} q^{7j} \frac{f_{49}^{4j-1}}{f_7^{4j-1}},\label{H4i-3}
\end{align}
and
\begin{equation}
H \left( q^{i} \frac{f_{7}^{4i}}{f_1^{4i}} \right)   = \sum_{j\geq1} m_{4i, i+j} \, q^{7j} \frac{f_{49}^{4j}}{f_7^{4j}}.\label{H4i} \\
\end{equation}
\end{lemma}
\begin{proof}
From \eqref{Gu1}, we have
\begin{equation*}
H \left( q^{2i} \frac{f_{49}^{i}}{f_1^{i}} \right) = \sum_{j\geq1} m_{i,j} \, q^{7j} \frac{f_{49}^{4j}}{f_7^{4j}}.
\end{equation*}
The fact that $m_{4i-1, j}=0$ for $1\leq j<i$ implies that
\begin{align*}
H \left( \left(q^{2} \frac{f_{49}}{f_1} \right)^{4i-1} \right) &=  \sum_{j\geq i} m_{4i-1,j} \, q^{7j} \frac{f_{49}^{4j}}{f_7^{4j}}\\
&=\sum_{j\geq 1} m_{4i-1,j+i-1} \, q^{7j+7i-7} \frac{f_{49}^{4j+4i-4}}{f_7^{4j+4i-4}},
\end{align*}
which yields \eqref{H4i-1}. Similarly, we can prove \eqref{H4i-3} and \eqref{H4i} follows from \eqref{Gu2}.
\end{proof}
\section{Generating functions}
In this section, we establish generating functions for $b_{7^{\ell}}(n)$ within specific arithmetic progressions.
\begin{theorem}\label{T1}
For each $\beta \geq0$ and $k\geq1$, we have
\begin{align}\label{G1}
 \nonumber   &\sum_{n\geq 0}b_{7^{2k-1}}\left(7^{2k+2\beta-1}n+\frac{18\cdot7^{2k+2\beta-1}-7^{2k-1}+1}{24}\right)q^n\\&=\sum_{j\geq 1}y^{(2k-1)}_{2\beta+1,j} \, q^{j-1}\left(\frac{f_7}{f_1}\right)^{4j-1}
\end{align}
and 
\begin{equation}\label{G2}
    \sum_{n\geq 0}b_{7^{2k-1}}\left(7^{2k+2\beta}n+\frac{6\cdot7^{2k+2\beta}-7^{2k-1}+1}{24}\right)q^n=\sum_{j\geq 1}y^{(2k-1)}_{2\beta+2,j} \,q^{j-1} \left(\frac{f_7}{f_1}\right)^{4j-3},
\end{equation}
where the coefficient vectors are defined as follows:
\begin{equation}\label{G11}
        y_{1,j}^{(2k-1)}=x_{2k-1,j}
\end{equation}
and
\begin{equation}\label{G12}
       y^{(2k-1)}_{\beta+1,j}= 
       \begin{cases}
      \displaystyle \sum_{i\geq1}y^{(2k-1)}_{\beta,i} m_{4i-1,i+j-1} \,\,\   & \text{if  $\beta$ is odd},\\
       \displaystyle \sum_{i\geq1}y^{(2k-1)}_{\beta,i}m_{4i-3,i+j-1} \,\,\   & \text{if  $\beta$ is even},
       \end{cases}
\end{equation}
for all $\beta, j\geq 1$.
\end{theorem}
\begin{proof}
From \eqref{bl}, we have
\begin{equation}\label{p1}
    \sum\limits_{n\geq0}b_{7^{2k-1}}(n)q^n=f_{7^{2k-1}}\sum\limits_{n\geq0}p(n)q^n.
\end{equation}
Extracting the terms involving $q^{7^{2k-1}n+\lambda_{2k-1}}$ on both sides of \eqref{p1} and dividing throughout by $q^{\lambda_{2k-1}}$, we obtain  
\begin{equation*}
    \sum\limits_{n\geq0}b_{7^{2k-1}}(7^{2k-1}n+\lambda_{2k-1})q^{7^{2k-1}n}=f_{7^{2k-1}}\sum\limits_{n\geq0}p(7^{2k-1}n+\lambda_{2k-1})q^{7^{2k-1}n}.
\end{equation*}
If we replace $q^{7^{2k-1}}$ by $q$ and  use \eqref{H1}, we get
\begin{equation*}
    \sum\limits_{n\geq0}b_{7^{2k-1}}(7^{2k-1}n+\lambda_{2k-1})q^n=\sum_{j\geq 1}x_{2k-1,j} \, q^{j-1}\frac{f_7^{4j-1}}{f_1^{4j-1}},
\end{equation*}
which is the case $\beta=0$ of \eqref{G1}.

We now assume that \eqref{G1} is true for some integer $\beta\geq 0$. Applying the operator $H$ to both sides, by \eqref{H4i-1}, we have

\begin{align*}
    &\sum_{n\geq 0}b_{7^{2k-1}}\left(7^{2k+2\beta}n+\frac{6\cdot7^{2k+2\beta}-7^{2k-1}+1}{24}\right)q^{7n}
    \\&=\sum_{i\geq 1}y^{(2k-1)}_{2\beta+1,i} \, H\left(q^{i-2}\frac{f_7^{4i-1}}{f_1^{4i-1}}\right)
    \\&=\sum_{i\geq 1}y^{(2k-1)}_{2\beta+1,i}\sum_{j\geq1} m_{4i-1,i+j-1}\, q^{7j-7} \frac{f_{49}^{4j-3}}{f_7^{4j-3}}
    \\&=\sum_{j\geq 1}\left(\sum_{i\geq1}y^{(2k-1)}_{2\beta+1,i} \, m_{4i-1,i+j-1}\right) q^{7j-7} \frac{f_{49}^{4j-3}}{f_7^{4j-3}}\\&
    =\sum_{j\geq 1}y^{(2k-1)}_{2\beta+2,j} \,q^{7j-7} \frac{f_{49}^{4j-3}}{f_7^{4j-3}}.
\end{align*}
That is,
\begin{equation*}
    \sum_{n\geq 0}b_{7^{2k-1}}\left(7^{2k+2\beta}n+\frac{6\cdot7^{2k+2\beta}-7^{2k-1}+1}{24}\right)q^n=\sum_{j\geq 1}y^{(2k-1)}_{2\beta+2,j} \,q^{j-1} \frac{f_{7}^{4j-3}}{f_1^{4j-3}}.
\end{equation*}
Hence if \eqref{G1} is true for some integer $\beta \geq 0$, then \eqref{G2} is true for $\beta$.
Suppose that \eqref{G2} is true for some integer $\beta \geq 0$. Applying the operating $H$ to both sides, by \eqref{H4i-3}, we obtain
\begin{align*}
       & \sum_{n\geq 0}b_{7^{2k-1}}\left(7^{2k+2\beta+1}n+\frac{18\cdot7^{2k+2\beta+1}-7^{2k-1}+1}{24}\right)q^{7n+7}
       \\&=\sum_{i\geq 1}y^{(2k-1)}_{2\beta+2,i} \sum_{j\geq1} m_{4i-3,i+j-1} q^{7j} \, \frac{f_{49}^{4j-1}}{f_7^{4j-1}}
       \\&=\sum_{j\geq 1}\left(\sum_{i\geq1}y^{(2k-1)}_{2\beta+2,i} \, m_{4i-3,i+j-1}\right) q^{7j} \, \frac{f_{49}^{4j-1}}{f_7^{4j-1}}
       \\&=\sum_{j\geq 1}y^{(2k-1)}_{2\beta+3,j} \,q^{7j} \, \frac{f_{49}^{4j-1}}{f_7^{4j-1}},
\end{align*}
which yields
\begin{equation*}
    \sum_{n\geq 0}b_{7^{2k-1}}\left(7^{2k+2\beta+1}n+\frac{18\cdot7^{2k+2\beta+1}-7^{2k-1}+1}{24}\right)q^{n+1}=\sum_{j\geq 1}y^{(2k-1)}_{2\beta+3,j} \,q^{j} \, \frac{f_{7}^{4j-1}}{f_1^{4j-1}}.
\end{equation*}
This is \eqref{G1} with $\beta$ replaced by $\beta+1$. This completes the proof.
\end{proof}
\begin{theorem} \label{T2}
For each $\beta \geq0$, and $k\geq1$, we have
 \begin{equation}\label{G3}
\sum_{n\geq 0}b_{7^{2k}}\left(7^{2k+\beta-1}\left(n+1\right)-\frac{7^{2k}-1}{24}\right)q^n=\sum_{i\geq 1}y^{(2k)}_{\beta+1,i} \, q^{i-1}\frac{f_7^{4i}}{f_1^{4i}}.
\end{equation}
where coefficient vectors are defined as follow:
\begin{equation*}
        y_{1,j}^{(2k)}=x_{2k-1,j}
\end{equation*}
and
\begin{equation*}
y^{(2k)}_{\beta+2,j}=\sum_{i\geq1}y^{(2k)}_{\beta+1,i}\, m_{4i,i+j}.
\end{equation*}
\end{theorem}
\begin{proof}
    Equation \eqref{bl} can be written as
\begin{equation*}
    \sum\limits_{n\geq0}b_{7^{2k}}(n)q^n=f_{7^{2k}}\sum\limits_{n\geq0}p(n)q^n.
\end{equation*}
Extracting the terms involving $q^{7^{2k-1}n+\lambda_{2k-1}}$ on both sides of the above equation and dividing throughout by $q^{\lambda_{2k-1}}$, we obtain  
\begin{equation}\label{p6}
    \sum\limits_{n\geq0}b_{7^{2k}}(7^{2k-1}n+\lambda_{2k-1})q^{7^{2k-1}}=f_{7^{2k}}\sum\limits_{n\geq0}p(7^{2k-1}n+\lambda_{2k-1})q^{7^{2k-1}}.
\end{equation}
If we replace $q^{7^{2k-1}}$ by $q$ and use \eqref{H1}, we get
\begin{equation}\label{p7}
    \sum\limits_{n\geq0}b_{7^{2k}}(7^{2k-1}n+\lambda_{2k-1})q^n=\sum_{j\geq 1}x_{2k-1,j} \, q^{j-1}\frac{f_7^{4j}}{f_1^{4j}},
\end{equation}
which is the case $\beta=0$ of \eqref{G3}.

We now assume that \eqref{G3} is true for some integer $\beta$. Applying the operator $H$ to both sides, by \eqref{H4i}, we have
\begin{equation*}
    \begin{split}
\sum_{n\geq0}b_{7^{2k}}\left(7^{2k+\beta}\left(n+1\right)-\frac{7^{2k}-1}{24}\right)q^{7n+7}&=\sum_{i\geq 1}y^{(2k)}_{\beta+1,i}\sum_{j\geq1} m_{4i, i+j} \, q^{7j} \frac{f_{49}^{4j}}{f_7^{4j}}\\&
        =\sum_{j\geq 1}\left(\sum_{i\geq1}y^{(2k)}_{\beta+1,i} \, m_{4i, i+j} \right)q^{7j} \frac{f_{49}^{4j}}{f_7^{4j}}\\&
        =\sum_{j\geq 1}y^{(2k)}_{\beta+2,j}\, q^{7j} \frac{f_{49}^{4j}}{f_7^{4j}},
    \end{split}
\end{equation*}
which implies that
\begin{equation*}
\sum_{n\geq0}b_{7^{2k}}\left(7^{2k+\beta}\left(n+1\right)-\frac{7^{2k}-1}{24}\right)q^n=\sum_{j\geq 1}y^{(2k)}_{\beta+2,j}q^{j-1} \frac{f_{7}^{4j}}{f_1^{4j}}.
\end{equation*}
So we obtain \eqref{G3} with $\beta$ replaced by $\beta+1$. 
\end{proof}

\section{Proof of Congruences}

For a positive integer $n$, let $\pi(n)$ be the highest power of $7$ that divides $n$, and define $\pi(0)=+\infty$.
\begin{lemma}[\cite{FG}, Lemma 5.1]
    For each $i,j\geq1 $, we have \begin{equation}\label{pmij} \pi\left(m_{i,j}\right)\geq\left[\frac{7j-2i-1}{4}\right].
    \end{equation}
     \end{lemma}
In particular,
\begin{align}
\pi\left(m_{4i-1,i+j-1}\right)&\geq\left[\frac{7(i+j-1)-2(4i-1)-1}{4}\right]=\left[\frac{7j-i-6}{4}\right],\label{p4i-1} \\
\pi\left(m_{4i-3,i+j-1}\right)&\geq\left[\frac{7(i+j-1)-2(4i-3)-1}{4}\right]=\left[\frac{7j-i-2}{4}\right],\label{p4i-3} 
\end{align}
and
\begin{equation}\label{p4i} 
\pi\left(m_{4i,i+j}\right)\geq\left[\frac{7(i+j)-2(4i)-1}{4}\right]=\left[\frac{7j-i-1}{4}\right].
\end{equation}
    
\begin{lemma}[\cite{FG}, Lemma 5.3]
    For all $k,j\geq1 $, we have
    \begin{equation}\label{pi1}
  \pi(x_{1,1})=1,\,\,\, \pi(x_{1,2}) =2,\,\,\, \pi\left(x_{2k+1,j}\right)\geq k+1+\left[\frac{7j-4}{4}\right]
     \end{equation} 
    and
    \begin{equation}\label{pi11}
   \pi\left(x_{2k,j}\right)\geq k+1+\left[\frac{7j-6}{4}\right].
     \end{equation}
    \end{lemma}
\begin{lemma}
    For all $j,k\geq1$ and $\beta\geq 0$, we have
\begin{equation}\label{pi4}
   \pi\left(y^{(2k-1)}_{2\beta+1,j}\right)\geq k+\beta+\left[\frac{7j-7}{4}\right]
\end{equation}
   and
\begin{equation}\label{pi5}
\pi\left(y^{(2k-1)}_{2\beta+2,j}\right)\geq k+\beta+\left[\frac{7j-7}{4}\right].
\end{equation}
\end{lemma}
\begin{proof}
In view of Theorem \ref{T1}, we have
\begin{equation*}
        y_{1,j}^{(2k-1)}=x_{2k-1,j}.
\end{equation*}
From \eqref{pi1}, we can see that 
\begin{equation*}
\pi\left(x_{2k-1,j}\right)\geq k+\left[\frac{7j-7}{4}\right],
\end{equation*}
this shows that \eqref{pi4} holds when $\beta=0$.

We now assume that \eqref{pi4} is true for some $\beta\geq 0$, then
\begin{align*}
   \pi\left(y^{(2k-1)}_{2\beta+2,j}\right)&\geq \min_{i\geq1}\left\{\pi\left(y^{(2k-1)}_{2\beta+1,i}\right)+\pi\left(m_{4i-1,i+j-1}\right)\right\}\\&
   \geq \min_{i\geq1}\left\{k+\beta+\left[\frac{7i-7}{4}\right]+\left[\frac{7j-i-6}{4}\right]\right\}\\&
   \geq k+\beta+\left[\frac{7j-7}{4}\right],
   \end{align*}
which is \eqref{pi5}. 

Now suppose \eqref{pi5} holds for some $\beta\geq0$. Then,
\begin{align*}
   \pi\left(y^{(2k-1)}_{2\beta+3,j}\right)&\geq \min_{i\geq1}\left\{\pi\left(y^{(2k-1)}_{2\beta+2,i}\right)+\pi\left(m_{4i-3,i+j-1}\right)\right\}\\&
   \geq \min_{i\geq1}\left\{k+\beta+\left[\frac{7i-7}{4}\right]+\left[\frac{7j-i-2}{4}\right]\right\}\\&
   \geq k+\beta+\left[\frac{7j-3}{4}\right]\\&
   = k+\beta+1+\left[\frac{7j-7}{4}\right],
\end{align*}
which is \eqref{pi4} with $\beta+1$ for $\beta$. This completes the proof.
\end{proof}
\begin{lemma}
For each $j,k\geq1$ and $\beta\geq 0$, we have
\begin{align}\label{pi6}
   \pi\left(y^{(2k)}_{\beta+1,j}\right)\geq k+\beta+\left[\frac{7j-7}{4}\right].
   \end{align}
\end{lemma}
\begin{proof}
In view of Theorem \ref{T2}, we have
\begin{equation*}
        y_{1,j}^{(2k)}=x_{2k-1,j}.
\end{equation*}
From \eqref{pi1}, we can see that \eqref{pi6} holds when $\beta=0$.
We now assume that \eqref{pi6} is true for some $\beta\geq 0$, then
\begin{align*}
   \pi\left(y^{(2k)}_{\beta+2,j}\right)&\geq \min_{i\geq1}\left\{\pi\left(y^{(2k)}_{\beta+1,i}\right)+\pi\left(m_{4i,i+j}\right)\right\}\\&
   \geq \min_{i\geq1}\left\{k+\beta+\left[\frac{7i-7}{4}\right]+\left[\frac{7j-i-1}{4}\right]\right\}\\&
   \geq k+\beta+\left[\frac{7j-2}{4}\right]\\&
   \geq k+\beta+1+\left[\frac{7j-7}{4}\right],
   \end{align*}
which is \eqref{pi6} with $\beta+1$ for $\beta$. This completes the proof.
\end{proof}
\noindent\textbf{Proof of Theorem \ref{th1}}
By \eqref{pi4} we see that $y_{2\beta+1, j}^{(2k-1)}\equiv 0 \pmod{7^{k+\beta}}$. So,  \eqref{c1} follows from \eqref{G1}.
Similarly, congruence \eqref{c3} follows from \eqref{G3} and \eqref{pi6}.

\end{document}